\documentclass{article}
\pdfoutput=1 
   \usepackage{authblk}
   \usepackage{amsthm}
   \usepackage{amsmath}
   \usepackage{amsfonts}    
   \usepackage{amssymb}    

	\newtheorem{theorem}{Theorem}[section]
	
	\newtheorem{lemma}{Lemma}[section]

	\newtheorem{algorithm}{Algorithm}[section]

	\newtheorem*{hfs}{Algorithm HFS}
	\newtheorem*{aplusatimesdata}{Algorithm $A^+AB$}
	\newtheorem*{aaplustimesdata}{Algorithm $AA^+B$}

	\theoremstyle{definition}

	\theoremstyle{remark}

	\numberwithin{equation}{section}

	\DeclareSymbolFont{boldoperators}{OT1}{cmr}{bx}{n}
	\SetSymbolFont{boldoperators}{bold}{OT1}{cmr}{bx}{n}
	\edef\bar{\unexpanded{\protect\mathaccentV{bar}}\number\symboldoperators16}

     \allowdisplaybreaks

\begin{document}


%
%
%
%
\def\put#1 at (#2,#3){\vbox to 0pt{\kern#3
                     \dimen2=#2
                     \hbox{\kern\dimen2{#1}}\vss}\nointerlineskip}

\def\dim{\hbox{Dim}}
\def\rank{\hbox{rank}}
\def\ret{\hfill \cr}
\def\qu#1{\hskip #1em\relax}
\def\Cal{\mathcal}
%
%
\def\bbc{\Bbb C}
%
%
\def\bbr{\Bbb R}
%
%
\def\intersect{\mathop{\cap}\displaylimits}
%
%
\def\sinc{\hbox{sinc}}
%
%
\def\twoprime{{\prime\prime}}
%
%
\def\threeprime{{\prime\prime\prime}}
%
%
\def\union{\mathop{\cup}\displaylimits}
%
%
\def\tten#1{$\times{10}^{#1}\qu{1}$}
%
%

\def\tcap#1#2{
\vskip -25pt
{\par\centering \small {\bf Table #1} #2\par}
\vskip 20pt
}
\def\tcapns#1#2{
{\par\centering \small {\bf Table #1} #2\par}
}


\def\smartqed{}

\newcommand{\sbt[1]}{\,\begin{picture}(-1,1)(-1,-#1)\circle*{#1}\end{picture}\ }

\newcommand{\abs}[1]{\lvert#1\rvert}

\newcommand{\blankbox}[2]{%
  \parbox{\columnwidth}{\centering
    \setlength{\fboxsep}{0pt}%
    \fbox{\raisebox{0pt}[#2]{\hspace{#1}}}%
  }%
}

\smartqed

\title{Compact Schemes for $A^+B$, $A^+AB$ and $AA^+B$}

\author{Marc Stromberg\break{\sl email: mstromberg@psmfc.org}}

\maketitle

\abstract{Explicit details are presented for calculation of $A^+B$, $A^+AB$ and $AA^+B$ where $A_{m\times n}$ is any nonzero matrix, $A^+$ is the Moore-Penrose pseudoinverse of $A$ and $B$ is any matrix of appropriate dimensions, where the quantities in question are found using only the storage originally allocated to the matrices $A$ and $B$ (together with some simple one dimensional indexing arrays). 

\vskip 10 pt
\noindent{\bf 2020 Mathematics Subject Classification: }  65F05, 65F45, 65F99 \hfill\break

\section{Introduction}
The calculations to be presented depend on the following theorem, which guarantees a specific rank decomposition of any nonzero $m \times n$ matrix. The proof of the theorem is given with minor modifications in Appendix A, as is a statement of the algorithm that produces the factorization of the theorem. The algorithm is typical of so-called compact schemes, in that the entries of the factors are stored by modifying those of $A$ in place. This in turn gives rise to the schemes of this paper, which are achieved by further in-place modifications of the original matrix.
\begin{theorem}
\label{th:lu1p1}\cite{ms-lufam}
Let $A_{m \times n}$ be a nonzero matrix of rank $r$. Then there are a permutation matrix $P_{m\times m}$, a lower trapezoidal matrix $L_{m\times r}$ of rank $r$ and an upper echelon matrix $U_{r\times n}$ of rank $r$ such that $PA = LU$.
\end{theorem}
The factorization of the theorem has numerous applications in the case that $A$ has reduced rank.  The compact scheme discussed here for $A^+A$ is used extensively in the computations for  \cite{ms-caqm}.  The factorization itself can be used for the construction of constraints that define the image of a polyhedron under a linear transformation, also in  \cite{ms-caqm}. The schemes we will detail involve further  in-place (post factorization) preparation of the matrix $A$ based on the fact that $A^+ =  U^*(UU^*)^{-1}(L^*L)^{-1}L^*P$, as is easily shown using properties of the pseudoinverse. 

The methods to be shown are based on modifications of the matrix $R$ which we define as the upper left  $r \times r$  submatrix of $PA$, and the submatrix of $A$ corresponding to $R$. The diagonal of $R$ consists of the pivot elements of the factor $L$. In the following we use the usual ${\cal L}{\cal D}{\cal L}^*$ variant of the Cholesky decomposition of a Hermitian positive definite matrix, expressed as such to avoid confusion with the factor $L$ of the $LU$ factorization. Sufficient detail is given for all of these schemes below that it is straightforward to convert any of them to a computer program.
\subsection{$A^+B$}
We first give the details for the calculation of $A^+B$ as outlined in \cite{ms-lufam} which will then be applied to the other two cases. We assume that the matrix $A$ has been factored according to Algorithm \ref{al:lu1p1} and that the integer arrays $\rho$ and $\gamma$ produced in the factorization have been retained. Here we assume that $B$ is any $m\times p$ matrix.

\item{\bf Step 1.}
The process of calculating $A^+B$ will overwrite portions of the factored form of $A$, so it is first necessary to calculate $C =L^*PB$. This is straightforward with, for $i=0,\dots,r-1$ and $q=0,\dots,p-1$, 
\begin{equation} C_{i, q} =\sum_{k=0}^{m-1}L_{k,i} B_{\rho_k,q}=
\sum_{k=i}^{m-1}\bar{A}_{\rho_k, \gamma_i} B_{\rho_k,q} \label{eq:ldl1p0}\end{equation}
However, it is clear that these results can be stored column by column back into a set of $r$ rows of each column of  $B$. Therefore, rather than introducing a new matrix $C$, we shall overwrite the appropriate entries of $B$ with the following assignments. Set
\begin{equation} B_{\rho_i, q} \leftarrow
\sum_{k=i}^{m-1}\bar{A}_{\rho_k, \gamma_i} B_{\rho_k,q} \label{eq:ldl1p01}\end{equation}
for $i=0,\dots,r-1$ and $q=0,\dots,p-1$. It is evident that entries to be overwritten are not needed in subsequent sums in (\ref {eq:ldl1p01}) provided the given order is followed. The entries $B_{\rho_i, q}$ will now contain $L^*PB$.

\item{\bf Step 2.}
Compute and store the lower triangular part of $L^*L$ into $A$. Explicitly, the first $r$ rows of $L$ comprise an $r\times r$ lower triangular matrix in $PA$, namely the lower triangular part of $R$. We overwrite this matrix with the lower triangular part of $L^*L$  by making the assignments \begin{equation}A_{\rho_j, \gamma_i} \leftarrow \sum_{k=j}^{m-1}A_{\rho_k, \gamma_j}\bar{A}_{\rho_k, \gamma_i} \label{eq:ldl1p1} 
\end{equation}
for $j=i,\dots,r-1$ for $i=0,\dots,r-1$. Note that these assignments do not disturb the contents of $U$ as stored in $PA$.

\item{\bf Step 3.}
Calculate $D = (L^*L)^{-1}C$. This is done by solving the system $$L^*L y^\prime = y$$ for each column $y^\prime$ of $D$ from each column $y$ of $C$. Specifically, this step is application of the following algorithm. Again, rather than introducing a new matrix $D$, we overwrite the appropriate entries of $B$. This is  an ${\cal L}{\cal D}{\cal L}^*$ decomposition (consisting of (\ref{eq:ldl1p2}),  (\ref{eq:ldl1p2a}), and  (\ref{eq:ldl1p2b})) followed by forward substitution (\ref{eq:ldl1p3}), application of the diagonal (\ref{eq:ldl1p4}), and back substitution (\ref{eq:ldl1p5}), which we state explicitly in terms of the factored contents of $PA$. Specifically, we find the ${\cal L}{\cal D}{\cal L}^*$ factorization of $L^*L$ from its lower triangular part, storing the factor ${\cal L}$ in the lower triangular $r\times r$ part of $R$ with the diagonal ${\cal D}$ overwriting the pivot elements of $L$, since the diagonal elements of ${\cal L}$ have the value $1$ and can be defined implicitly.
\begin{hfs} 
\label{al:ltfd} Let $A$ be the result of the factorization of Algorithm \ref{al:lu1p1} and let $\rho$ and $\gamma$ be the integer arrays of length at least $m$ and $n$ respectively also resulting from the factorization.
Assume that the $r \times r$ lower triangular part of $R$ has been overwritten by the lower triangular part T of an $r\times r$ Hermitian positive definite matrix $H$ (e.g., as in Step 2). 
\end{hfs}
{
\setlength{\jot}{4pt}
\begin{align}
&\hbox{for} \quad 0\leq i < r \quad \{\notag \\
&\qu{2.2}\hbox{for}  \quad 0\leq k < i \quad\{\notag\\
&\qu{2.2}\quad A_{\rho_i,\gamma_i}\leftarrow  A_{\rho_i,\gamma_i} -  A_{\rho_i,\gamma_k} \bar{A}_{\rho_i,\gamma_k} A_{\rho_k,\gamma_k}\label{eq:ldl1p2}\\
&\qu{2.2}\}\qu{3.2}\notag\\
&\qu{2.2}\hbox{for}  \quad i+1\leq j < r \quad\{\notag\\
&\qu{4.2}\hbox{for}  \quad 0\leq k < i \quad\{\notag\\
&\qu{4.2}\quad A_{\rho_j,\gamma_i}\leftarrow  A_{\rho_j,\gamma_i} -   A_{\rho_j,\gamma_k}\bar{A}_{\rho_i,\gamma_k} A_{\rho_k,\gamma_k}\label{eq:ldl1p2a}\\
&\qu{4.2}\}\qu{3.2}\notag\\
&\qu{3.2}\quad A_{\rho_j,\gamma_i}\leftarrow  A_{\rho_j,\gamma_i} /  A_{\rho_i,\gamma_i} \label{eq:ldl1p2b}\\
&\qu{2.2}\}\qu{3.2}\notag\\
&\}\notag\\
&\hbox{for} \quad 0\leq q < p \quad \{\notag \\
&\qu{2.2}\hbox{for} \quad 0\leq i < r \quad \{\notag \\
&\qu{4.2}\hbox{for}  \quad 0\leq k < i \quad\{\notag\\
&\qu{4.2}\quad B_{\rho_i,q} \leftarrow B_{\rho_i,q} -  A_{\rho_i,\gamma_k}B_{\rho_k,q} \label{eq:ldl1p3}\\
&\qu{4.2}\}\notag\\
&\qu{2.2}\}\notag\\
&\qu{2.2}\hbox{for} \quad 0\leq i < r \quad \{\notag \\
&\qu{2.2}\quad B_{\rho_i,q} \leftarrow B_{\rho_i,q} / A_{\rho_i,\gamma_i} \label{eq:ldl1p4}\\
&\qu{2.2}\}\qu{4.2}\notag\\
&\qu{2.2}\hbox{for} \quad r> i \geq 0 \quad \{\notag \\
&\qu{4.2}\hbox{for}  \quad r > k > i \quad\{\notag\\
&\qu{4.2}\quad B_{\rho_i,q} \leftarrow B_{\rho_i,q} -  \bar{A}_{\rho_k,\gamma_i}B_{\rho_k,q} \label{eq:ldl1p5}\\
&\qu{4.2}\}\qu{3.2}\notag\\
&\qu{2.2}\}\notag\\
&\}\qu{2.2}\notag
\end{align}
}
\vskip 1pt
\noindent

\item{\bf Step 4.}
Compute and store the lower triangular part of $UU^*$ into $A$, again overwriting the upper left $r\times r$ lower triangular matrix in $PA$ . We overwrite this matrix with the lower triangular part of $UU^*$  by making the assignments 
\begin{equation} A_{\rho_j, \gamma_i} \leftarrow \begin{cases} \bar{A}_{\rho_i, \gamma_j} & \hbox{ if } j > i\cr
							     1 & \hbox{ if } j = i\cr \end{cases} + \sum_{k=1+\gamma_j}^{n-1}
A_{\rho_j, k}\bar{A}_{\rho_i, k} \label{eq:ldl1p6}
\end{equation}
 for $j=i,\ldots,r-1$ for $i=0,\ldots,r-1$. 

\item{\bf Step 5.} 
Calculate $F =  (UU^*)^{-1}D$ by again applying Algorithm HFS,  where of course $F$ now consists of the entries $B_{\rho_i, q}$ for $i=0,\dots,r-1$ and $q=0,\dots,p-1$.

\item{\bf Step 6.} 
Complete the calculation of $A^+B = G = U^*F $ with the assignments
\begin{equation}G_{iq} = 
\sum_{k=0}^{r-1} B_{\rho_k,q}\begin{cases} 1 & \text{if $i = \gamma_k$}\\
\bar{A}_{\rho_k, i} & \text{if $i > \gamma_k$}\\
0 & \text{if $i < \gamma_k$}
\end{cases} \label{eq:ldl1p7}
\end{equation} for $i=0,\dots,n-1$ and $q=0,\dots,p-1$.

Calculation of the remaining quantities is for each split into two phases, namely preparation of the matrix $A$ and then application of the result to the data $B$.
\subsection{$A^+AB$}
For this calculation we again assume that the matrix $A$ has been factored according to Algorithm \ref{al:lu1p1} and that the integer arrays $\rho$ and $\gamma$ produced in the factorization have been retained. Here we assume that $B$ is any $n\times p$ matrix.
Steps 1 and 2 prepare the contents of $A$, the results of which are reusable with arbitrary $B$. Step 3 is application of $A^+A$ to $B$.
\item{\bf Step 1.}  Preparation of $A$ is continued from an initial $LU$ factorization by first replacing the upper left $r\times r$ lower triangular matrix of $PA$ with the lower triangular part of  $UU^*$ as in (\ref{eq:ldl1p6}). 

\item{\bf Step 2.}  Compute the ${\cal L}{\cal D}{\cal L}^*$ decomposition of $UU^*$ by applying only the factorization (\ref{eq:ldl1p2}), (\ref{eq:ldl1p2a}), and (\ref{eq:ldl1p2b}) of Algorithm HFS. The calculation of $A^+AB$ does not require retention of $L$ of the factorization, so the scheme is free to overwrite that portion of the factored form of $A$.  

\item{\bf Step 3.}  Use the following algorithm to apply $A^+A$ to arbitrary data $B$. Simply put, this computes $U^*(UU^*)^{-1}UB$.
\begin{aplusatimesdata} Let $A$ be the result of the factorization of Algorithm \ref{al:lu1p1} and let $\rho$ and $\gamma$ be the integer arrays of length at least $m$ and $n$ respectively also resulting from the factorization.
Assume that the upper left $r \times r$ lower triangular part of $R$ has been overwritten by the lower triangular $r\times r$ factor T of the of the ${\cal L}{\cal D}{\cal L}^*$ factored form of $UU^*$  and let $B$ be an arbitrary $n\times p$ matrix.
\end{aplusatimesdata}
First calculate $C = UB$ as follows in (\ref{eq:ldl1pc1}), then apply forward substitution (\ref{eq:ldl1pc2}), application of the diagonal (\ref{eq:ldl1pc3}) and back substitution (\ref{eq:ldl1pc4}) to obtain $D = (UU^*)^{-1}C$, and finally $A^+AB = G = U^*D$ (\ref{eq:ldl1pc5}). Again rather than introducing new matrices $C$ and $D$, we overwrite the appropriate elements of $B$.

{
\setlength{\jot}{4pt}
\begin{align}
&\hbox{for} \quad 0\leq q < p \quad \{\notag \\
&\qu{2.2}\hbox{for} \quad 0\leq i < r \quad \{\notag \\
&\qu{2.2}\quad B_{\gamma_i, q} \leftarrow B_{\gamma_i, q} + \sum_{j=\gamma_i+1}^{n-1}
A_{\rho_i, j}B_{jq} \label{eq:ldl1pc1}\\
&\qu{2.2}\}\notag\\
&\qu{2.2}\hbox{for} \quad 0\leq i < r \quad \{\notag \\
&\qu{2.2}\quad B_{\gamma_i, q} \leftarrow  B_{\gamma_i, q} - \sum_{k=0}^{i-1}A_{\rho_i, \gamma_k}B_{\gamma_k, q} \label{eq:ldl1pc2}\\
&\qu{2.2}\}\notag\\
&\qu{2.2}\hbox{for} \quad 0\leq i < r \quad \{\notag \\
&\qu{2.2}\quad  B_{\gamma_i, q} \leftarrow   B_{\gamma_i, q}/A_{\rho_i, \gamma_i} 
\label{eq:ldl1pc3}\\
&\qu{2.2}\}\notag\\
&\qu{2.2}\hbox{for} \quad r > i \geq 0 \quad \{\notag \\
&\qu{2.2}\quad B_{\gamma_i, q} \leftarrow B_{\gamma_i, q} - \sum_{k=i+1}^{r-1}
\bar{A}_{\rho_k, \gamma_i}B_{\gamma_k, q} \label{eq:ldl1pc4}\\
&\qu{2.2}\}\notag\\
&\qu{2.2}\hbox{for} \quad  0\leq i < n \quad \{\notag \\
&\qu{2.2}\quad G_{iq} \leftarrow  \sum_{k=0}^{r-1} B_{\gamma_k, q}\begin{cases} 1 & \text{if $i = \gamma_k$}\\
\bar{A}_{\rho_k, i} & \text{if $i > \gamma_k$}\\
0 & \text{if $i < \gamma_k$}
\end{cases} \label{eq:ldl1pc5}\\
&\qu{2.2}\}\notag\\
&\}\qu{2.2}\notag
\end{align}
}

\subsection{$AA^+B$}
We again assume that the matrix $A$ has been factored according to Algorithm \ref{al:lu1p1} and that the integer arrays $\rho$ and $\gamma$ produced in the factorization have been retained. Here we assume that $B$ is any $m\times p$ matrix. Steps 1--3 below prepare the contents of $A$, and again the results are reusable with arbitrary $B$. Step 4 is the application of $AA^+$ to $B$. First, note that 
\begin{equation}AA^+ = P^*L(L^*L)^{-1}L^*P\label{eq:ldl1pc6}\end{equation} and furthermore that (\ref{eq:ldl1pc6}) remains true if $L$ is replaced by $LQ$ where $Q$ is any $r\times r$ invertible matrix. For the current application we will in particular take $Q$ to be the diagonal of reciprocals of the pivot elements of $L$, namely the diagonal matrix with entries $\{1/A_{\rho_i, \gamma_i}\}$ for $i=0,\dots,r-1$.
\item{\bf Step 1.}  Preparation of $A$ is continued from an initial $LU$ factorization by first replacing $L$ with $LQ$ in $PA$ as follows.
{
\setlength{\jot}{4pt}
\begin{align}
&\hbox{for} \quad 0\leq i < r \quad \{\notag \\
&\qu{2.2}\hbox{for} \quad i+1\leq k < m \quad \{\notag \\
&\qu{2.2}\quad A_{\rho_k, \gamma_i} \leftarrow A_{\rho_k, \gamma_i} /
A_{\rho_i, \gamma_i} \label{eq:ldl1pc7}\\
&\qu{2.2}\}\notag\\
&\}\qu{2.2}\notag
\end{align}
}
\item{\bf Step 2.} 
We will henceforth assume that $L$ refers to the result of the replacements (\ref {eq:ldl1pc7}). Since we now can assume that the pivot elements of $L$ (diagonal of $R$) all have the value $1$, we can define these implicitly and are free to overwrite them, and thus to overwrite the upper triangular part of $R$ without disturbing the contents of $L$. To this end we replace the upper triangular part of $R$ with the upper triangular part of $L^*L$ as follows. Though this is stored as an upper triangular matrix, it can be treated as a lower triangular matrix for further applications, as explained below.
{
\setlength{\jot}{4pt}
\begin{align}
&\hbox{for} \quad 0\leq i < r \quad  \{\notag \\
&\qu{2.2}\hbox{for} \quad i\leq j < r \quad \{\notag \\
&\qu{2.2}\quad A_{\rho_i, \gamma_j} \leftarrow  \sum_{k=j}^{m-1}
A^\prime_{\rho_k, \gamma_i}\bar{A}^\prime_{\rho_k, \gamma_j} 
\label{eq:ldl1pc8}\\
&\qu{2.2}\}\notag\\
&\}\qu{2.2}\notag
\end{align}
}
where we define \begin{equation}
A_{\rho_k, \gamma_q}^\prime =\begin{cases} A_{\rho_k, \gamma_q} &\text{if  $k>q$}\cr
1 &\text{if  $k=q$}\cr 0 &\text{if  $k<q$}\cr\end{cases}
\end{equation}
for $k=0,\dots,m=1$ and $q=0,\dots,r-1$.
\item{\bf Step 3.} The last step of preparation of $A$ is to factor the matrix stored in Step 2 as though it is lower triangular and storing the transpose of the lower triangular factor into the upper triangular part of $R$ as follows.	
{
\setlength{\jot}{2.5pt}
\begin{align}
&\hbox{for} \quad 0\leq i < r \quad  \{\notag \\
&\qu{2.2}\hbox{for} \quad 0\leq k < i \quad \{\notag \\
&\qu{2.2}\quad A_{\rho_i, \gamma_i} \leftarrow  A_{\rho_i, \gamma_i}-
A_{\rho_k, \gamma_i}\bar{A}_{\rho_k, \gamma_i}A_{\rho_k, \gamma_k}
\label{eq:ldl1pc9}\\
&\qu{2.2}\}\notag\\
&\qu{2.2}\hbox{for} \quad i+1\leq j < r \quad \{\notag \\
&\qu{4.2}\hbox{for}\quad 0\leq k < i\notag \\
&\qu{4.2}\quad A_{\rho_i, \gamma_j} \leftarrow  A_{\rho_i, \gamma_j}-
A_{\rho_k, \gamma_j}\bar{A}_{\rho_k, \gamma_i}A_{\rho_k, \gamma_k}
\label{eq:ldl1pc10}\\
&\qu{4.2}\}\qu{2.2}\notag\\
&\qu{4.2}A_{\rho_i, \gamma_j} \leftarrow  A_{\rho_i, \gamma_j}/
A_{\rho_i, \gamma_i}\notag \\
&\qu{2.2}\}\qu{2.2}\notag\\
&\}\qu{2.2}\notag
\end{align}
}
\item{\bf Step 4.} 
Use the following algorithm to apply $AA^+$ to arbitrary data $B$. This computes $P^*L(L^*L)^{-1}L^*PB$,  where $L$ is the result of the replacements (\ref{eq:ldl1pc7}).
\begin{aaplustimesdata} Let $A$ be the result of the factorization of Algorithm \ref{al:lu1p1} and let $\rho$ and $\gamma$ be the integer arrays of length at least $m$ and $n$ respectively also resulting from the factorization.
Assume that the $r \times r$ upper triangular part of $R$ has been overwritten by the upper triangular $r\times r$ factor T from (\ref{eq:ldl1pc9}), (\ref{eq:ldl1pc10}) of the of the ${\cal L}{\cal D}{\cal L}^*$ factored form of $L^*L$  and let $B$ be an arbitrary $m\times p$ matrix.
\end{aaplustimesdata}
First calculate $C = L^*PB$ as follows in (\ref{eq:ldl1pd1}), then apply forward substitution (\ref{eq:ldl1pd2}), application of the diagonal (\ref{eq:ldl1pd3}) and back substitution (\ref{eq:ldl1pd4}) to obtain $D = (L^*L)^{-1}C$, and finally $AA^+B = G =P^*LD$ (\ref{eq:ldl1pd5}). Again we merely update the elements of $B$ rather than introducing new matrices $C$ and $D$.
 {
\setlength{\jot}{4pt}
\begin{align}
&\hbox{for} \quad 0\leq q < p \quad \{\notag \\
&\qu{2.2}\hbox{for} \quad 0\leq i < r \quad \{\notag \\
&\qu{2.2}\quad B_{\rho_i, q} \leftarrow B_{\rho_i, q} + \sum_{k=i+1}^{m-1}
\bar{A}_{\rho_k, \gamma_i}B_{\rho_k,q} \label{eq:ldl1pd1}\\
&\qu{2.2}\}\notag\\
&\qu{2.2}\hbox{for} \quad 0\leq i < r \quad \{\notag \\
&\qu{2.2}\quad B_{\rho_i, q} \leftarrow B_{\rho_i, q} - \sum_{k=0}^{i-1}A_{\rho_k, \gamma_i}B_{\rho_k, q} \label{eq:ldl1pd2}\\
&\qu{2.2}\}\notag\\
&\qu{2.2}\hbox{for} \quad 0\leq i < r \quad \{\notag \\
&\qu{2.2}\quad B_{\rho_i, q} \leftarrow  B_{\rho_i, q}/A_{\rho_i, \gamma_i} 
\label{eq:ldl1pd3}\\
&\qu{2.2}\}\notag\\
&\qu{2.2}\hbox{for} \quad r > i \geq 0 \quad \{\notag \\
&\qu{2.2}\quad B_{\rho_i, q} \leftarrow  B_{\rho_i, q} - \sum_{k=i+1}^{r-1}\bar{A}_{\rho_i, \gamma_k}B_{\rho_k, q} \label{eq:ldl1pd4}\\
&\qu{2.2}\}\notag\\
&\qu{2.2}\hbox{for} \quad  0\leq i < m \quad \{\notag \\
&\qu{2.2}\quad G_{\rho_i,q} \leftarrow  \sum_{k=0}^{r-1} B_{\rho_k, q}\begin{cases} 
A_{\rho_i, \gamma_k} & \text{if $i > k$}\\
1 & \text{if $i = k$}\\
0 & \text{if $i < k$}
\end{cases} \label{eq:ldl1pd5}\\
&\qu{2.2}\}\notag\\
&\}\qu{2.2}\notag
\end{align}
}

\section{Example}
We end this discussion by giving a simple example. Computations were done in double precision, with the results displayed here to $5$ places. Starting with the rank $4$ matrix
\begin{equation}
A_{5x7} = \left[\begin{matrix} 1&2&3&4&5&6&7\cr
 7&6&5&4&3&2&1\cr
 1&2&3&4&3&2&1\cr
 1&7&1&7&1&7&1\cr
 7&1&7&1&7&1&7\cr \end{matrix}\right] \notag
\end{equation}
we find its $LU$ factorization according to Algorithm \ref{al:lu1p1}, with 
\begin{equation}
L = \left[\begin{matrix} 7&0\phantom{---}&0\phantom{---}&0\cr
1&6.14286&0\phantom{---}&0\cr
1&1.14286&2.23256&0\cr
1&1.14286&2.23256&2\cr
7&-5\phantom{----}&2.23256&2\cr
 \end{matrix}\right]\notag
\end{equation}
and
\begin{equation}
U = \left[\begin{matrix} 1&0.85714&0.71429&0.57143&0.42857&0.28571&0.14286\cr
0&1\phantom{---}&0.04651&1.04651&0.09302&1.09302&0.13953\cr
0&0\phantom{---}&1\phantom{---}&1\phantom{---}&1.10417&0.20833&0.31250\cr
0&0\phantom{---}&0\phantom{---}&0\phantom{---}&1\phantom{---}&2\phantom{---}&3\phantom{---}\cr \end{matrix}\right] \notag
\end{equation}
with row permutation $(0,1,2,3,4)\rightarrow(1,3,2,0,4)$.
Applying the scheme of 1.1, we compute $A^+B$ with $B$ as the identity matrix $I_{5x5}$, which gives
\begin{equation}
A^+ = \left[\begin{matrix} -0.02388&\phantom{-}0.08326&-0.15&\phantom{-}0.01719&\phantom{-}0.04219\cr
-0.01071&\phantom{-}0.06071&-0.10&\phantom{-}0.05833&-0.00833\cr
-0.03192&\phantom{-}0.00379&\phantom{-}0.15&-0.04323&\phantom{-}0.01510\cr
-0.01875&-0.01875&\phantom{-}0.20&-0.00208&-0.03542\cr
\phantom{-}0.00379&-0.03192&\phantom{-}0.15&-0.04323&\phantom{-}0.01510\cr
\phantom{-}0.06071&-0.01071&-0.10&\phantom{-}0.05833&-0.00833\cr 
\phantom{-}0.08326&-0.02388&-0.15&\phantom{-}0.01719&\phantom{-}0.04219\cr
\end{matrix}\right]. \notag
\end{equation} 
We next apply the scheme 1.2, computing $A^+AA^+$ taking $B=A^+$ and then scheme 1.3, computing $AA^+A$ taking $B=A$ (the original $A$). The maximum difference between entries of $A$ and $AA^+A$ and between entries of $A^+$ and $A^+AA^+$ is found to be less than $2.7\text{e-}15$. The results are identical whether using the fine, coarse or simple (with $\varepsilon = 1.0\text{e-}12$) methods from Appendix A.
\vfill\break

\appendix 
\counterwithin*{equation}{section}
\renewcommand\theequation{\thesection\arabic{equation}}
\section{$LU$ Factorization}
The proof of Theorem \ref{th:lu1p1} depends on a pair of lemmas that concern the construction of $P$, $L$ and $U$ in Algorithm \ref{al:lu1p1}. The array $\rho$ will contain the permutation information, so that $(PA)_{i, j} = A_{\rho_i, j}$. 
For the following we note that the changes to $A$ (the assignments (\ref{eq:lu1p2}), (\ref{eq:lu1p5})) take place only once for any position. 
\begin{algorithm} 
\label{al:lu1p1}
Let $\rho$ and $\gamma$ be integer arrays of length at least $m$ and $n$ respectively, and let $D$ be a real array of length $m$. Let $r$ and $p$ be integers and set $r \leftarrow 0$ and $p \leftarrow 0$. For $0 \leq i < m$ set $\rho_i \leftarrow i$ and $D_i \leftarrow \|A_{i,\cdot}\|$ where $\| \|$ is a norm on the rows of $A$. Let $R$ be an array of integers of length at least $n$.
\end{algorithm}
{
\setlength{\jot}{5pt}
\begin{align}
\hbox{for}& \quad 0\leq \ell < n \quad \{ \label{eq:lu1p0}\\
&\hbox{set}\quad C \leftarrow 0 \notag\\
&\hbox{for}\quad r\leq i < m \quad \{\label{eq:lu1p1}\\
&\qu{1.2}\hbox{if} \quad D_{\rho_i} > 0 \quad \{ \notag\\
&\qu{2.4}A_{\rho_i,\ell} \leftarrow A_{\rho_i,\ell} - \sum_{0\leq k< r}A_{\rho_i,\gamma_k}A_{\rho_k,\ell} \label{eq:lu1p2} \\
&\qu{2.4}\hbox{if} \quad C < |A_{\rho_i,\ell}|/D_{\rho_i} \quad \hbox{set} \quad C \leftarrow |A_{\rho_i,\ell}|/D_{\rho_i}, \quad p \leftarrow i \notag\\
&\qu{1.2}\}\notag\\
&\}\notag\\
&\hbox{if}  \quad C > 0 \quad \{ \label{eq:lu1p3} \\
&\qu{1.2}\hbox{set} \quad \gamma_r \leftarrow \ell, \quad \rho_p \leftrightarrow \rho_r \notag\\
&\qu{1.2}\hbox{for} \quad \ell+1\leq j < n \quad \{ \label{eq:lu1p4}\\
&\qu{2.4}A_{\rho_r,j} \leftarrow (A_{\rho_r,j}-\sum_{0\leq k<r} A_{\rho_r,\gamma_k}A_{\rho_k,j})/A_{\rho_r,\ell} \label{eq:lu1p5}\\
&\qu{1.2}\} \notag\\
&\qu{1.2}\hbox{set} \quad r \leftarrow r+1 \notag\\
&\} \notag\\
&\hbox{set} \quad R_\ell \leftarrow r-1 \label{eq:lu1p6}\\
\}\qu{.8}&\notag
\end{align}
}

\vskip 1pt
\noindent
We define $L$ and $U$ in terms of the eventual value of $A$ in Algorithm \ref{al:lu1p1} by

\begin{equation}
U_{p, q}  =
\begin{cases} 1 & \text{if $q = \gamma_p$}\\
A_{\rho_p, q} & \text{if $q > \gamma_p$}\\
0 & \text{if $q < \gamma_p$}
\end{cases}\label{eq:lu1p7}
\end{equation}
 for $p = 0,...r-1$ and $q = 0,...,n-1$ and
\begin{equation}L_{p, q}  = \begin{cases}  A_{\rho_p, \gamma_q} &
\text{ if  $p \geq q$}\\
0 &\text{ if  $p < q$}\end{cases}\label{eq:lu1p8}
\end{equation}
for $p = 0,...m-1$ and $q = 0,...,r-1$. These definitions are assumed to apply as soon as the assignments in (\ref{eq:lu1p2}) and (\ref{eq:lu1p5}) have taken place for the relevant entries of $A$.

Theorem \ref{th:lu1p1} is established by examination of the algorithm and applying the following two lemmas.
We assume the usual convention that vacuous sums are zero.
By observation, the algorithm will proceed through the outer loop (\ref{eq:lu1p0}), terminating after at most $n$ steps. The vector $R$ of the algorithm is a construct for use only in the lemmas, and can be omitted from implementations. 

\begin{lemma} 
\label{le:lu1p1} \cite{ms-lufam}
Let $\ell_0$ be the smallest value of $\ell$ for which (\ref{eq:lu1p3}) of Algorithm \ref{al:lu1p1} is satisfied. Then for each $\ell$, $\ell_0\leq \ell <n$, we have 
$0\leq R_\ell <m$ and $\gamma_{R_\ell} \leq \ell$. For any such $\ell$ let $\ell^\prime = \ell + 1$. If $R_\ell = R_{\ell^\prime}$ then $\gamma_{R_{\ell^\prime}} < \ell^\prime$ and if $R_\ell < R_{\ell^\prime}$ then 
$R_{\ell^\prime} = R_{\ell}+1$ and $\gamma_{R_{\ell^\prime}} = \ell^\prime$.
Furthermore either
$\gamma_{R_\ell} < \ell$, in which case we have
 \begin{equation}(PA)_{i,\ell} = (LU)_{i,\ell}\label{eq:lu1p9}\end{equation} for $i = R_\ell +1, \ldots, m-1$, 
 or  $\gamma_{R_\ell} = \ell$ and we have
 \begin{equation}(PA)_{i,\ell} = (LU)_{i,\ell}\label{eq:lu1p10}\end{equation} for $i = R_\ell, \ldots, m-1$ and 
 \begin{equation}(PA)_{i,j} = (LU)_{i,j}\label{eq:lu1p11}\end{equation} for $i = R_\ell$ and $j=\ell+1, \ldots, n-1$.
 In any case, 
 \begin{equation}(PA)_{i,j} = (LU)_{i,j} \label{eq:lu1p12}\end{equation} for $i=0,\ldots,m-1$ and $j=0, \ldots, \ell_0-1$.
\end{lemma}
We give a slightly modified proof of the following. 
\begin{lemma} 
\label{le:lu1p2} \cite{ms-lufam}
For each $\ell = 0, \ldots, n-1$ we have
\begin{equation}(PA)_{i, q} = (LU)_{i, q}\label{eq:lu1p31}\end{equation} for $i = 0,\ldots, m-1$ and $q = 0, \ldots, \ell$.  If  $\ell \geq \ell_0$ of Lemma \ref{le:lu1p1} then

\begin{equation}(PA)_{p, j} = (LU)_{p, j}\label{eq:lu1p32}\end{equation} for $p = 0,\ldots, R_\ell$ and $j = \ell, \ldots, n-1$.
\end{lemma} 
\begin{proof} We have (\ref{eq:lu1p31}) for $\ell < \ell_0$ by Lemma \ref{le:lu1p1}, so suppose $\ell \geq \ell_0$. Then $R_\ell \leq \ell$ and $\gamma_{R_\ell}\leq \ell$, also by Lemma \ref{le:lu1p1} and its proof. Suppose the present lemma is true for some $\ell \geq 0$ and let $\ell^\prime = \ell+1$. If $R_\ell = R_{\ell^\prime}$ then $\gamma_{R_{\ell^\prime}} < \ell^\prime$. By the previous lemma, $(PA)_{i, \ell^\prime} = (LU)_{i, \ell^\prime}$ for $i = R_{\ell^\prime}+1, \ldots, m-1$. By assumption we have $(PA)_{p, j} = (LU)_{p, j}$ for $p=0, \ldots, R_\ell$ and $j=\ell, \ldots, n-1$, which shows that $(PA)_{i,\ell^\prime} = (LU)_{i,\ell^\prime}$ for $i=0, \ldots, R_{\ell^\prime}$ thus $(PA)_{i, \ell^\prime} = (LU)_{i, \ell^\prime}$ for $i = 0, \ldots, m-1$ and 
$(PA)_{p, j} = (LU)_{p, j}$ for $p=0, \ldots, R_{\ell^\prime}$ which is the induction hypothesis for $\ell^\prime$.

On the other hand if $R_\ell < R_{\ell^\prime}$ then $R_{\ell^\prime}=R_\ell + 1$ and $\gamma_{R_{\ell^\prime}} = \ell^\prime$. Then
$(PA)_{i, \ell^\prime} = (LU)_{i, \ell^\prime}$ for $i = R_{\ell^\prime}, \ldots, m-1$. Again by assumption we have 
$(PA)_{p, j} = (LU)_{p, j}$ for $p=0, \ldots, R_\ell$ and $j=\ell, \ldots, n-1$, which shows that 
$(PA)_{i, \ell^\prime} = (LU)_{i, \ell^\prime}$ for $i = 0,\dots,R_{\ell^\prime}-1$ thus
$(PA)_{i, \ell^\prime} = (LU)_{i, \ell^\prime}$ for $i = 0,\dots,m-1$. We have 
$(PA)_{p, j} = (LU)_{p, j}$ for $p=0, \ldots, R_\ell$ and $j=\ell, \ldots, n-1$ by assumption and 
$(PA)_{p, j} = (LU)_{p, j}$ for $p=R_{\ell^\prime}$ and $j=\ell^\prime+1,\ldots,n-1$ by Lemma \ref{le:lu1p1}, and for $j = \ell^\prime$ by choosing $i = R_{\ell^\prime}$ above so that 
$(PA)_{p, j} = (LU)_{p, j}$ for $p=0, \ldots, R_{\ell^\prime}$ and $j=\ell^\prime, \ldots, n-1$. This satisfies the inductive hypothesis.
Now for $\ell_0$, (\ref{eq:lu1p10}) and (\ref{eq:lu1p11}) hold, and we note that $R_{\ell_0} = 0$. Then (\ref{eq:lu1p31}) is true for $q = \ell_0$, and we have (\ref{eq:lu1p31}) for $q<\ell_0$ by (\ref{eq:lu1p12}), so (\ref{eq:lu1p31}) holds for $i=0,\ldots,m-1$ and $q=0, \ldots, \ell_0$. By (\ref{eq:lu1p11}) we have (\ref{eq:lu1p32}) for $j>\ell_0$. By (\ref{eq:lu1p10}) we have (\ref{eq:lu1p32}) for $j=\ell_0$. This instance completes the induction and proof of Lemma \ref{le:lu1p2}.
\end{proof}

\noindent {\it Proof of Theorem \ref{th:lu1p1} } By Lemma \ref{le:lu1p2} it is clear that $PA = LU$. That $L$ is lower trapezoidal (that is, a truncated lower triangular matrix) and $U$ is upper echelon and that both have rank $r$ follows by construction.\qed\linebreak

\noindent {\it Note on machine arithmetic.} Applications of the factorization will depend on use of machine numbers and in practice this means that the comparison (\ref{eq:lu1p3}) can be misleading in the sense that a very small floating point number may either represent a value that is supposed to be zero or a nonzero value that just happens to be small. The value of $C$ in the algorithm is the maximum of the nonzero values $|A_{\rho_i,\ell}|/D_{\rho_i}$ if there is one, or is zero if not. Therefore it is desirable to have a scheme for deciding which of these quantities is nonzero as opposed to being an artifact of machine arithmetic that appears nonzero. A number of methods can be used to separate these possibilities, one of which might be to simply find $C$ as the maximum of the required machine values and to then stipulate that $C>\varepsilon$ (for some well chosen $\varepsilon$) should be interpreted as that (\ref{eq:lu1p3}) is in fact satisfied. This will be referred to as the \textbf{\textit{simple}} method, and we will now discuss some adaptations of that approach. According to the results of \cite{jr-iebipfpa}, any order of evaluation of inner products returns a machine number $\hat \nu$ such that
\begin{equation}|{\hat \nu} - x^Ty| \leq nu|x|^T|y|\label{eq:lu1p13}\end{equation} where $x$ and $y$ are vectors of length $n$ and $u$ is the machine unit.

Using (\ref{eq:lu1p13}) it is easy to show that if $\nu_a$ is the machine-computed value of $|x|^T|y|$, and if 
\begin{equation}|{\hat \nu}| > \frac{nu}{1-nu} \nu_a  \label{eq:lu1p13a}\end{equation}
then the exact arithmetic value of $x^Ty$ is nonzero (where we assume that $nu<1$). It should be noted that if (\ref {eq:lu1p13}) is true, it is also true replacing $n$ by $K$ where $K$ is the number of pairs $(x_i, y_i)$ for which both components are nonzero, provided that the computation of ${\hat \nu}$ ignores pairs for which one component is zero (that is, the floating point representation of zero), so that (\ref {eq:lu1p13}) and (\ref {eq:lu1p13a}) will at least sometimes be overestimates. 
In any case, for machine computations we find $C$ by considering  values of $|A_{\rho_i,\ell}|/D_{\rho_i}$ that are actually nonzero by using  (\ref {eq:lu1p14}), since by (\ref{eq:lu1p13a}) the exact arithmetic value of the right side of the assignment (\ref {eq:lu1p2}) is nonzero if 
\begin{equation}|A_{\rho_i,\ell}| > \phi(K) (|A_{\rho_i,\ell}|+\sum_{0\leq k<r} |A_{\rho_i,\gamma_k}||A_{\rho_k,\ell}|)\label{eq:lu1p14}\end{equation} 
where we define $\phi(a)  = {au}/{(1-au)}$ and where $K$ is determined in the process of evaluating the right side of (\ref {eq:lu1p14}) as the count of nonzero terms in the sum, plus $1$ if $A_{\rho_i,\ell} \neq 0$.
The value to the left of the inequality in (\ref {eq:lu1p14}) is after the assignment (\ref {eq:lu1p2}) (i.e., the computed value),  values to the right of the inequality are values before the assignment and all expressions in (\ref {eq:lu1p14}) represent machine-computed values. 
Inequality  (\ref {eq:lu1p14}) applies for real computations. If $x$ and $y$ are complex, it suffices to consider the real and imaginary parts of $x^*y$ separately. Let $x_R$ and $x_I$ represent the vectors of real and imaginary parts of $x$ and define  $y_R$ and $y_I$ similarly. Then $\Re{x^*y} = x_R^Ty_R + x_I^Ty_I =(x_R, x_I)^T (y_R, y_I)$ where if $a, b$ are vectors of length $n$,  $(a, b)$ represents the vector of length $2n$ consisting of the entries of $a$ followed by the entries of $b$. It follows that if ${\hat \nu}$ is the machine value of $x^*y$ then 
\begin{equation} |\Re ({\hat \nu} - {x^*y})|\leq 2nu (|x_R|, |x_I|)^T (|y_R|, |y_I|)\label{eq:lu1p15}\end{equation} from (\ref {eq:lu1p13}) with similar results for imaginary parts. Letting ${\tilde \nu}_R$ and 
 ${\tilde \nu}_I$ represent the machine computed values of $ (|x_R|, |x_I|)^T (|y_R|, |y_I|)$ and $ (|x_R|, |x_I|)^T (|y_I|, |y_R|)$ respectively, we have similarly to (\ref {eq:lu1p13a}) that if either
\begin{equation}
|\Re {\hat \nu}| > \phi(2n) {\tilde \nu}_R \quad\text{or}\quad |\Im {\hat \nu}| >  \phi(2n) {\tilde \nu}_I \label{eq:lu1p16}
\end{equation} then $x^*y \neq 0$. The alternatives (\ref {eq:lu1p16}) can be used to produce a pair of inequalities similar to (\ref{eq:lu1p14}) for the complex case, namely

\begin{equation}\begin{aligned} 
|{\Cal R}_{\rho_i, \ell}| > \phi_{2r}
(|{\Cal R}_{\rho_i, \ell}|+&\sum_{0\leq k< r}|{\Cal R}_{\rho_i, k}||{\Cal R}_{\rho_k, \ell}| + |{\Cal I}_{\rho_i, k}||{\Cal I}_{\rho_k, \ell}| ) \label{eq:lu1p17}\\
|{\Cal I}_{\rho_i, \ell}| > \phi_{2r}
(|{\Cal I}_{\rho_i, \ell}|+&\sum_{0\leq k< r}|{\Cal R}_{\rho_i, k}||{\Cal I}_{\rho_k, \ell}| + |{\Cal I}_{\rho_i, k}||{\Cal R}_{\rho_k, \ell}| )\\
\end{aligned} \end{equation}
where $\phi_{2r} = \phi(2r+1)$ 
and we define ${\Cal R}_{ij} = \Re A_{ij}$, and ${\Cal I}_{ij} = \Im A_{ij}$. A variation of (\ref {eq:lu1p17}) can also be applied that counts nonzero terms in the sums.
Application of (\ref{eq:lu1p14}) or the variant (\ref {eq:lu1p17})  (both of which will be called \textbf{\textit{fine}} methods) will introduce extra floating point operations. 
One approach to reduce the number of operations is to define $\mu$ as the initial maximum value of $|A_{ij}|$ for $i=0,\dots,m-1$ and $j=0,\dots,n-1$ and $\mu_R$, $\mu_I$ as  ${\max\limits_{i,j}} |\Re A_{ij}|$ and  ${\max\limits_{i,j}} |\Im A_{ij}|$ respectively, and to use (with $\kappa = \min \{m, n\}$)
\begin{equation}|A_{\rho_i,\ell}| > \phi(\kappa +1) (\mu+ \kappa\mu^2)\label{eq:lu1p119}\end{equation}  or in the complex case
\begin{equation}\begin{aligned} 
|\Re A_{\rho_i, \ell}| >& \phi(2\kappa +1)  
(\mu_R + \kappa\mu_R^2 +\kappa \mu_I^2)\cr  
|\Im A_{\rho_i, \ell}| >& \phi(2\kappa +1) 
(\mu_I +2 \kappa\mu_I\mu_R)\cr
\end{aligned} \label{eq:lu1p19}\end{equation}
calculating the value of $\phi$ only once in either case and updating and retaining the entire right hand sides of the respective inequalities (\ref{eq:lu1p119}) or (\ref{eq:lu1p19}) only when $\mu$, $\mu_R$, or $\mu_I$ requires updating after (\ref{eq:lu1p2}) or  (\ref{eq:lu1p5}).
The inequalities  (\ref{eq:lu1p119}) or (\ref{eq:lu1p19}) will be called \textbf{\textit{coarse}} methods because they are overestimates compared to (\ref{eq:lu1p14}) and (\ref {eq:lu1p17}). 

The simple method is typically the fastest, but requires a judicious choice of $\varepsilon$. Numerical experiments for real arithmetic indicate that on a modern processor on average and independent of the size of the matrix $A$, the fine method takes about $1.1$ times as long as the simple method and the coarse method is between the simple and fine methods but very nearly equal to the simple method in execution time.

Finally, it is possible for any of these methods to be wrong (especially the simple with $\varepsilon > 0$ or coarse methods, the fine method being somewhat more adaptive in this situation) by not concluding that $C>0$, if $A$ has many elements that are very small (within a few orders of magnitude of the machine unit) in absolute value. In that case, the algorithm will still produce a factorization, but the rank will be too small and the difference between $PA$ and $LU$ will effectively be those small elements of $A$. Whether or not this matters depends on the application. This situation might possibly be avoided by preparing $A$ ahead of time by scaling or another transformation. One can also do this deliberately, forcing the situation by using a constant larger than the machine unit $u$ in the above methods, in an attempt to filter {\lq noise\rq} from $A$. At the other extreme, there is the simple method with $\varepsilon = 0$ i.e., the original algorithm, which treats every value not equal to floating point zero as nonzero.

\end{document}